\documentclass[a4paper,12pt]{amsart}

\usepackage{t1enc}
\usepackage{float}
\restylefloat{table}
\def\I{\mathbb{I}}

\def\R{\mathbb{R}}
\def\T{\mathbb{T}}

\def\X{\mathbb{X}}
\def\Y{\mathbb{Y}}
\def\Z{\mathbb{Z}}

\def\cD{\mathcal{D}}
\def\cE{\mathcal{E}}

\def\al{\alpha}

\def\de{\delta}

\def\si{\sigma}

\def\om{\omega}
\def\Ga{\Gamma}
\def\De{\Delta}

\def\La{\Lambda}
\def\Si{\Sigma}

\def\Om{\Omega}
\def\Up{\Upsilon}

\def\P{\mathrm{P}}

\DeclareMathOperator\erf{erf}

\newcommand{\rpl}                         
{\mbox{$
\begin{picture}(12.7,8)(-.5,-1)
\put(0,0.2){$+$}
\put(4.2,2.8){\oval(8,8)[r]}
\end{picture}$}}

\newtheorem{theorem}{Theorem}[section]
\newtheorem{lemma}[theorem]{Lemma}
\newtheorem{proposition}[theorem]{Proposition}

\theoremstyle{remark}

\newtheorem*{remark*}{\rm\bf Remark}
\newtheorem{definition}[theorem]{\bf Definition}
\newtheorem{example}[theorem]{\bf Example}

\usepackage{amssymb, stmaryrd}
\usepackage{amscd}

\newcommand{\nd}{\nabla}

\makeatletter
\newcommand{\semidownbracefill}{$\m@th\braceld\leaders\vrule\hfill\braceru
  \bracelu\leaders\vrule\hfill\arrowhead$}



\def\sideremark#1{\ifvmode\leavevmode\fi\vadjust{\vbox to0pt{\vss
 \hbox to 0pt{\hskip\hsize\hskip1em
 \vbox{\hsize3cm\tiny\raggedright\pretolerance10000
 \noindent #1\hfill}\hss}\vbox to8pt{\vfil}\vss}}}%
                        
   \subjclass[2010]{Primary 53A30, 53B15; Secondary 35N10}

   \keywords{Conformal differential geometry, Tractor bundles, Einstein's equations}

\author{Matthew Randall}
 \address{Leibniz Universit\"at Hannover, Institut f\"ur Differentialgeometrie, Welfengarten 1, D-30167 Hannover}
\email{matthew.randall@math.uni-hannover.de}
\title{The conformal-to-Einstein equation on M\"obius surfaces}

\begin{document}

\begin{abstract}
 On a conformal manifold, it is well known that parallel sections of the standard tractor bundle with non-vanishing scale are in 1-1 correspondence with solutions of the conformal Einstein equation. In $2$ dimensions conformal geometry carries no local information but one can remedy this by equipping the surface with a M\"obius structure. The conformal Einstein equation is then a well defined overdetermined system of linear PDEs on the M\"obius surface. It turns out however that an additional term involving curvature in the form of the Cotton-York tensor appears in the prolongation of the conformal Einstein equation in $2$-dimensions, in contrast to the higher dimensional setting. We establish a 1-1 correspondence between solutions of the conformal Einstein equation on M\"obius surfaces with non-vanishing scale and parallel sections of the modified standard tractor connection called the prolongation connection. We also discuss the consequence of this in deriving local obstructions, and also obtain (partial) results on the possible dimensions of the kernel space of the conformal-to-Einstein operator on M\"obius surfaces.    

\end{abstract}

\maketitle

\pagestyle{myheadings}
\markboth{Randall}{The conformal-to-Einstein equation on M\"obius surfaces}

\section{Introduction}
Let $(M^2,[g])$ be a Riemann surface. This is a smooth $2$-dimensional oriented manifold equipped with a conformal structure $[g]$, which is an equivalence class of smooth Riemannian metrics under the equivalence relation $g_{ab}\mapsto \widehat g_{ab}=\Om^2 g_{ab}$ for any smooth positive nowhere vanishing function $\Om$. Since every metric $g_{ab}$ in dimension $2$ is locally a conformal rescaling of the flat metric $\de_{ab}$, conformal geometry in dimension $2$ carries no local information. To remedy this, one can impose on Riemann surfaces additional local structure present in conformal manifolds of dimension $n>2$. This is the motivation behind M\"obius structures \cite{mob}. A Riemann surface with a M\"obius structure will henceforth be called a M\"obius surface. On M\"obius surfaces the conformal-to-Einstein equation makes sense, and can be seen as a specialisation of the equation in conformal geometry in higher dimensions to the $2$-dimensional setting. The conformal-to-Einstein equation on conformal manifolds has been studied extensively, for instance in \cite{AlmObst} and \cite{GoNur}. Local algebraic obstructions to the existence of conformally Einstein metrics have also been found, for example in dimension $4$ the Bach tensor is one such obstruction. In higher dimensions these local obstructions are studied in \cite{GoNur}. Also other obstructions are found for self-dual manifolds in \cite{twist} and a complete set of obstructions is found recently in \cite{duntod}. In this paper we derive algebraic constraints for a given non-flat M\"obius surface to admit a solution to the conformal-to-Einstein equation, and from there derive obstructions to existence of solutions in the non-flat setting. Checking that the obstructions do not vanish tells us definitively that the M\"obius surface cannot admit any conformally Einstein metrics locally. We first recall facts about M\"obius surfaces in Section \ref{mobius}. In Section \ref{MEW} we discuss how the conformal Einstein equation on M\"obius surfaces is related to the scalar-flat M\"obius Einstein-Weyl (sf-MEW) equation, studied by the same author in \cite{sfmew}. In Section \ref{ceprolong} we prolong the conformal-to-Einstein equation on M\"obius surfaces to form a closed system and proceed to derive local obstructions in Sections \ref{ceobst} and \ref{nongeneric}. The relationship between the conformal-to-Einstein equation and tractor calculus is explored in Section \ref{tractorcal}. Examples are given in Sections \ref{ceobst} and \ref{nongeneric}. We give partial results on the dimension of the kernel space of the conformal-to-Einstein operator in Section \ref{dimker} and conclude in Section \ref{outlook} with an outlook on future research directions. Parts of this paper also appear in the author's thesis \cite{thesis}. Abstract indices \cite{pr} will be used throughout the paper to describe tensors on the conformal manifold. We have already used $g_{ab}$ to denote the metric tensor. For another instance, if we write $\om_{a}$ to denote a smooth $1$-form $\om$, then the $2$-form $d\om$ can be written as $\nd_{[a}\om_{b]}=\frac{1}{2}(\nd_a\om_b-\nd_b\om_a)$. For a representative metric $g_{ab}$ of the conformal class $[g]$, let $\nd$ denote its associated Levi-Civita connection. Under a conformal rescaling of the metric $\widehat g_{ab}=\Om^2 g_{ab}$, we have
\[
\widehat \nd_a\om_b=\nd_a\om_b-\Up_a\om_b-\Up_b\om_a+g_{ab}\Up^c\om_c
\]
for $\Up_a=\nd_a\log \Om$, where $\widehat \nd$ is the Levi-Civita connection for the conformally rescaled metric $\widehat g_{ab}$.

\section{M\"obius surfaces}\label{mobius}
A M\"obius surface is a Riemann surface $(M^2,[g])$ equipped with a smooth M\"obius structure as defined in $2.1$ of \cite{mob}. Following the conventions in \cite{mob}, in $2$ dimensions, the line bundle of conformal densities of weight $-2$, denoted by $\cE[-2]$, is defined to be the bundle of volume forms $\La^2TM$. The density line bundle of weight $1$ denoted by $\cE[1]$ (where it is called $L^1$ in \cite{mob}), is then given by the root bundle $(\La^2TM)^{-\frac{1}{2}}$. A Weyl derivative $D_a$ is defined to be a covariant derivative on $\cE[1]$. It is shown in \cite{mob} that the induced conformal change in the trace-free part of the Hessian defined by a Weyl derivative on densities of weight $1$ is tensorial. Taking the Weyl derivative to $D_a$ to be the one induced by a Levi-Civita connection $\nd_a$ for a particular representative metric $g_{ab}$ in the conformal class $[g]$, this motivates the following definition: 
\begin{definition}{(\cite{mob})}
A M\"obius structure on $(M^2,[g])$ is a smooth second order linear differential operator $D_{ab}: \cE[1] \to \cE_{(ab)_\circ}[1]$ such that $D_{ab}-\nd_{(a}\nd_{b)_\circ}$ is a zero order operator. 
\end{definition}

Let $\P_{(ab)_\circ}$ be the symmetric trace-free tensor denoting the difference, i.e.\
\begin{align*}
\P_{(ab)_\circ}\si:=(D_{ab}-\nd_{(a}\nd_{b)_\circ})\si,
\end{align*}  
where $\si$ is a section of $\cE[1]$. Since the operator $D_{ab}=D_{(ab)_\circ}$ is invariantly defined, under a conformal rescaling of the metric $\widehat g_{ab}=\Om^2g_{ab}$ we find that 
\begin{align*}
\widehat \P_{(ab)_\circ}=\P_{(ab)_\circ}-\nd_a\Up_b+\Up_a\Up_b-\frac{1}{2}g_{ab}\Up_c\Up^c+\frac{1}{2}g_{ab}\nd_c\Up^c 
\end{align*} 
where $\Up_a=\nd_a\log \Om$. Let $K$ denote the Gauss curvature of $g_{ab}$, i.e.\ $K=\frac{R}{2}$, where $R$ is the scalar curvature of $g_{ab}$. Define the Rho tensor by $\P_{ab}:=\P_{(ab)_\circ}+\frac{K}{2}g_{ab}$. Under a conformal rescaling, $K$ transforms as $\widehat K=K-\nd_a\Up^a$ and therefore 
\begin{align*}
\widehat \P_{ab}=\P_{ab}-\nd_a\Up_b+\Up_a\Up_b-\frac{1}{2}g_{ab}\Up_c\Up^c.
\end{align*} 
Hence for any representative metric $g_{ab}$ in the conformal class $[g]$ with its associated Levi-Civita connection $\nd_a$, a M\"obius structure in the sense of \cite{mob} determines a symmetric tensor $\P_{ab}$ satisfying the following two properties: \\
1) The metric trace of $\P_{ab}$ is the Gauss curvature $K$ of $g_{ab}$; \\
2) Under a conformal rescaling of the metric $\hat g_{ab}=\Om^2g_{ab}$, the tensor $\P_{ab}$ transforms accordingly as
\begin{equation}\label{Rho}
\widehat \P_{ab}=\P_{ab}-\nd_a\Up_b+\Up_a\Up_b-\frac{1}{2}g_{ab}\Up_c\Up^c,
\end{equation}     
where $\Up_a=\nd_a\log \Om$. Hence one may equivalently define:
\begin{definition}{(\cite{mobreview} and \cite{sfmew})} A M\"obius structure $[\P]$ on $(M^2,[g])$ is the assignment of a smooth symmetric tensor $\P_{ab}$ representative of $[\P]$ to each metric $g_{ab}$ in the conformal class, such that $g^{ab}\P_{ab}=K$ and the conformal transformation law (\ref{Rho}) is satisfied. 
\end{definition}
A M\"obius surface will be denoted by $(M^2,[g],[\P])$, and note that a representative in the class $[\P]$ is dependent on a representative of the conformal class $[g]$.  
Unlike in higher dimensions, the Schouten tensor obtained from the Riemann curvature tensor is not well-defined in dimension 2 and a M\"obius structure remedies that by equipping the manifold with a Rho tensor $\P_{ab}$ that behaves like a Schouten tensor under conformal rescaling. 
This allows us to write 
\begin{equation}\label{curvid}
R_{abcd}=K(g_{ac}g_{bd}-g_{bc}g_{ad})\equiv \P_{ac}g_{bd}-\P_{bc}g_{ad}+\P_{bd}g_{ac}-\P_{ad}g_{bc},
\end{equation}
even though the tensor $\P_{ab}$ cannot be recovered from the Riemannian curvature tensor alone, in contrast to the higher dimensional setting. A fixed representative metric $g_{ab}$ from the conformal class $[g]$ can be viewed as having conformal weight $2$ and induces a volume form $\epsilon_{ab}=\epsilon_{[ab]}$ of conformal weight $2$. We set our convention so that $\epsilon^{ab}\epsilon_{cb}=\de_c{}^a$ and we raise and lower indices using the metric. The Cotton-York tensor given by
\[
Y_{abc}=\nd_a\P_{bc}-\nd_b\P_{ac}
\]
is a M\"obius invariant of the M\"obius structure. This means that under conformal rescalings of the metric, the quantity $Y_{abc}$ remains unchanged. 
We can use the volume form $\epsilon_{ab}$ to dualise, so that
\[
Y_{abc}=\frac{1}{2}\epsilon_{ab}Y_c, 
\]
where $Y_c=\epsilon^{ab}Y_{abc}$
is now a $1$-form of conformal weight $-2$. 
Observe that
\[
\frac{1}{2}\epsilon_{ab}Y^b=Y_{ab}{}^b=\nd_aK-\nd^b\P_{ab}.
\]
A M\"obius surface is called flat iff $Y_a=0$ and not flat otherwise. In \cite{mob} the vanishing of $Y_a$ is shown to be equivalent to integrability of the M\"obius structure.  

\section{Conformal-to-Einstein equation on M\"obius surfaces}\label{MEW}
In this section we derive the conformal-to-Einstein equation on M\"obius surfaces. Let $(M^2,[g],[\P])$ be a M\"obius surface. 
A Weyl connection $D_a$ on a Riemann surface $(M^2,[g])$ is a torsion-free connection that preserves the conformal class $[g]$, or equivalently $D_a g_{bc}=2\al_ag_{bc}$ for some $1$-form $\al_a$. The $1$-form $\al_a$ is determined up to a gauge freedom; under a conformal rescaling of the metric $g_{ab}\mapsto \widehat g_{ab}=\Om^2 g_{ab}$, we have $\al_a \mapsto \widehat \al_a=\al_a+\Up_a$, where again $\Up_a=\nd_a\log \Om$. A Weyl connection $D_a$ is called closed if the $1$-form $\al_a$ determined by the Weyl derivative is closed, i.e\ $\nd_{[a}\al_{b]}=F_{ab}=0$. In this case the Weyl connection is locally a Levi-Civita connection for a metric in $[g]$. For a fixed M\"obius surface $(M^2,[g],[\P])$, we can ask whether there is a compatible Weyl connection $D_a$ such that the second order linear differential operator $D_{ab}$ is given by the trace-free symmetric Hessian of the Weyl derivative, i.e.\ $D_{ab}=D_{(a}D_{b)_\circ}$. For a representative metric $g_{ab} \in [g]$ and its associated Levi-Civita connection $\nd_a$, this is equivalent to solving the system of equations given by
\begin{equation}\label{mew}
\mbox{Trace-free part of }(\nd_{(a}\al_{b)}+\al_a\al_b+\P_{ab})=0.
\end{equation}
The conformal-to-Einstein equation is obtained from (\ref{mew}) by imposing the additional condition that $\al_a$ is closed (equivalently we want a compatible Weyl connection $D_a$ that is closed). Choosing a representative metric $g_{ab}$ in the conformal class $[g]$ with its associated Levi-Civita connection $\nd_a$, if $F_{ab}=\nd_{[a}\al_{b]}=0$ we can write the conformal-to-Einstein equation as 
\begin{align}\label{mewf=0} 
\nd_a\al_b+\al_a\al_b+\P_{ab}-\frac{\al_c\al^c}{2}g_{ab}=-\frac{s}{2}g_{ab},
\end{align}
where $s$ here is the scalar curvature of the Weyl connection $D_a$. Equation (\ref{mew}) is introduced in \cite{thesis} (where it is called the M\"obius Einstein-Weyl (MEW) equation) and mentioned in \cite{sfmew}. It is found to be not finite type in the sense of \cite{spencer} but becomes finite type when we additionally impose $F=\epsilon^{ab}F_{ab}=0$, in which case we obtain the conformal-to-Einstein equation, or when we additionally impose $s=0$, which is studied in \cite{sfmew} as the sf-MEW equation. Equation (\ref{mew}) generalises the conformal-to-Einstein equation on M\"obius surfaces the same way that the Einstein-Weyl equation generalises the conformal-to-Einstein equation in higher dimensions. Together with \cite{sfmew}, this paper examines the cases where equation (\ref{mew}) specialises to finite type system of PDEs on M\"obius surfaces. In summary, we have a study of the various different types of overdetermined systems of PDEs on M\"obius surfaces: 
\pagebreak
\begin{table}[h!]
  \centering
  \begin{tabular}{|c|c|c|}
    \cline{2-3}
    \multicolumn{1}{c|}{} & $F=0$ & $F\neq 0$ \\ \hline
    $s=0$ &  \mbox{Conformal-to-Einstein on flat surfaces}  & \mbox{sf-MEW}    \\ \hline
    $s\neq 0$ & \mbox{Conformal-to-Einstein}   & \mbox{MEW (not finite type)}     \\ \hline
  \end{tabular} 
\end{table}\\
The conformal-to-Einstein equation on flat surfaces (both $F=0$ and $s=0$ case) is dealt with in Subsection \ref{flatsol}. A more direct and heuristic definition for the conformal-to-Einstein condition is to ask for a fixed representative Rho tensor $\P_{ab}$ of the M\"obius structure whether it is Einstein in the sense that it is some multiple of a representative metric $g_{ab}$ in the conformal class $[g]$. The conformal-to-Einstein equation can be then similarly derived this way. 

\section{Prolongation of the conformal-to-Einstein equation on M\"obius surfaces}\label{ceprolong}
In this section we apply the prolongation procedure to the linear system of PDEs associated to the conformal-to-Einstein condition on M\"obius surfaces and see that the closed system differs from that obtained in the higher dimensional setting. Prolongation of the conformal-to-Einstein equation in higher dimensions is well known (see \cite{BEGo} and \cite{GoNur}). The closed system we obtain has consequences in deriving algebraic obstructions for the existence of conformally Einstein metrics on general non-flat M\"obius surfaces.
Let $(M^2,[g],[\P])$ be a M\"obius surface. The conformal-to-Einstein equation makes sense, and is given by
\begin{align}\label{ce0} 
(\nd_{a}\nd_{b}+\P_{ab})_{\circ}\si=0,
\end{align}
where $(\ldots)_\circ$ denotes trace-free part. This is obtained from equation (\ref{mewf=0}) by substituting $\al_a=\frac{\nd_a\si}{\si}$. Rewriting (\ref{ce0}) by introducing the trace term $\La$, we obtain
\begin{align}\label{ce}
\nd_a\nd_b\si+\P_{ab}\si+\La g_{ab}=0.
\end{align}
Equation (\ref{ce}) is conformally invariant when $\si$ has conformal weight $1$. Alternatively we can view solutions $\si$ satisfying (\ref{ce0}) as lying in the kernel of the conformal-to-Einstein operator $\cD_{ab}: \cE[1] \mapsto \cE_{(ab)_\circ}[1]$ given by
\begin{equation}\label{ceop}
\cD_{ab}\si=\mbox{Trace-free part of }(\nd_a\nd_b+\P_{ab})\si. 
\end{equation}
We shall prolong the conformal-to-Einstein equation and find remarkably that it differs from the prolongation in higher dimensions with the addition of a curvature term!
Let $\mu_a=\nd_a\si$. Then (\ref{ce}) is
\begin{equation}\label{cen}
\nd_a\mu_b=-\P_{ab}\si-\La g_{ab}. 
\end{equation}
Differentiating we obtain
\[
\nd_c\nd_a\mu_b+(\nd_c\P_{ab})\si+\P_{ab}\mu_c+\nd_c\La g_{ab}=0,
\]
so that skewing $a$ and $c$ indices, 
\begin{align*}
R_{ac}{}^d{}_b\mu_d+Y_{cab}\si+\P_{ab}\mu_c-\P_{cb}\mu_a+\nd_c\La g_{ab}-\nd_a\La g_{cb}=0.
\end{align*} 
Using that on M\"obius surfaces the curvature decomposes according to (\ref{curvid}), we obtain
\begin{align*}
\P_a{}^d\mu_dg_{cb}-\P_c{}^d\mu_dg_{ab}+Y_{cab}\si+\nd_c\La g_{ab}-\nd_a\La g_{cb}=0.
\end{align*} 
Tracing $c$ and $b$ indices, we get 
\begin{align*}
\nd_a\La=-Y_{ac}{}^c\si+\P_a{}^d\mu_d.
\end{align*} 
Now in dimensions $n>2$, $Y_{ac}{}^c=0$ because the Cotton-York tensor is totally trace-free. However, in $2$ dimensions,
\[
Y_{abc}=\frac{1}{2}\epsilon_{ab}Y_c,
\]  
where $Y_c=\epsilon^{ab}Y_{abc}$, so that we have
\[
Y_{ac}{}^c=\frac{1}{2}\epsilon_{ac}Y^c=\frac{1}{2}U_a, 
\]
where $U_a=\epsilon_{ac}Y^c$ is introduced and used in \cite{sfmew}. Hence we obtain
\begin{align}\label{ce3}
\nd_a\La=-\frac{1}{2}U_a\si+\P_a{}^d\mu_d.
\end{align} 
The conformal-to-Einstein equation prolongs to form a closed system, which defines a connection on the tractor bundle associated to a M\"obius surface. We remark that the prolongation of this equation also appears in some earlier form in (4.6) of \cite{OS}.

\section{Tractor calculus on M\"obius surfaces}\label{tractorcal}
In this section we shall discuss how to construct the conformal standard tractor bundle $\cE_A$ on M\"obius surfaces. The exposition on the tractor bundle in conformal geometry can be found in \cite{BEGo} and \cite{GoNur}. Because of the way we defined M\"obius surfaces, the conformal class $[g]$ on $M^2$ is assumed to be Riemannian and so of signature $(2,0)$. However, the tractor construction in this section essentially goes through without change for conformal classes with metrics of signature $(1,1)$ or $(0,2)$. Also the construction of the bundle only involves the conformal structure $[g]$ and the M\"obius structure $[\P]$ enters through defining the tractor connection. 
We use the tractor bundle here to really mean its dual, the co-tractor bundle, under the identification via the tractor metric $h_{AB}$. This essentially follows the construction given in \cite{BEGo}. The jet exact sequence at the $2$-jets of the density line bundle $\cE[1]$ gives
\begin{equation}\label{seq2jet}
0\to \cE_{(ab)}[1]\to J^2(\cE[1])\to J^1(\cE[1])\to 0,
\end{equation}
and the conformal structure further decomposes $\cE_{(ab)}[1]$ into the direct sum $\cE_{(ab)_\circ}[1]\oplus\cE[-1]$. The symmetric trace-free bundle $\cE_{(ab)_\circ}[1]$ is a smooth subbundle of $J^2(\cE[1])$, and the tractor bundle $\cE_A$ is simply the quotient bundle of $J^2(\cE[1])$ by $\cE_{(ab)_{\circ}}[1]$, defined by the exact sequence
\begin{equation}\label{ctrdef}
0\to \cE_{(ab)_\circ}[1]\to J^2(\cE[1])\to \cE_A\to 0.
\end{equation}
The short exact sequence (\ref{seq2jet}) at the 2-jets level and the short exact sequence
\[
0\to \cE_{a}[1]\to J^1(\cE[1])\to \cE[1]\to 0
\]
at the 1-jet level determine a composition series for $\cE_A$ described by 
\[
\cE_A= \cE[-1]\rpl\cE_a[1]\rpl\cE[1]
\]
where $\rpl$ is the semi-direct sum. A choice of metric $g_{ab} \in [g]$ determines a splitting of the exact sequence, and identifies the standard tractor bundle $\cE_A$ with the direct sum $\cE[-1]\oplus \cE_a[1]\oplus\cE[1]$. On M\"obius surfaces, the conformal cotractor bundle has an invariant metric $h_{AB}$ of signature $(2,1)$ called the tractor metric, and an invariant connection $\nabla_a$ preserving $h_{AB}$ called the tractor connection. Moreover, we have a section of the standard tractor bundle $\T_A\in \Ga\cE_A$ given in a conformal scale obtained by choosing a representative metric $g \in [g]$ by
\begin{align*} 
\T_A\stackrel{g}=\begin{pmatrix}
\si \\ \mu_b \\ \La
\end{pmatrix}\in \begin{matrix}
\cE[1]\\ \oplus \\
\cE_a[1]\\ \oplus\\
\cE[-1]
\end{matrix},
\end{align*}
and under conformal rescalings of the metric, 
\begin{align*} 
\widehat \T_A\stackrel{\widehat g}=\begin{pmatrix}
\si \\ \mu_b+\Up_b\si \\ \La-\Up^b\mu_b-\frac{1}{2}\Up_b\Up^b\si
\end{pmatrix}.
\end{align*}
Using the tractor bases convention as used in \cite{GoNur} for instance, we have 
\begin{align*}
\T_A\stackrel{g}{=}\si \Y_A+\mu_a \Z^a_A+\La \X_A,  
\end{align*}
where under conformal recalings, the bases transform according to 
\begin{align*}
\widehat \Y_A=\Y_A-\Up_a\Z^a_A-\frac{1}{2}\Up_a\Up^a\X_A,\qquad \widehat \Z^a_A=\Z^a_A+\Up^a\X_A,\qquad \widehat \X_A=\X_A,
\end{align*} 
and the tractor connection acts on the bases according to 
\begin{align*}
\nd_a \Y_A=\P_{ab}\Z^b_A,\qquad \nd_a \Z_{bA}=-\P_{ab}\X_A-g_{ab}\Y_A,\qquad \nd_a \X_A=\Z_{aA}.
\end{align*}
Note that there is a choice in defining the tractor connection but this is motivated by the higher dimensional setting.  
On a M\"obius surface $(M^2,[g],[\P])$, we have constructed the standard tractor connection given by 
\begin{align*}
\nd_a\T_B=
\nd_a\begin{pmatrix}
\si \\ \mu_b \\ \La
\end{pmatrix}
=\begin{pmatrix}
\nd_a\si-\mu_a \\ \nd_a\mu_b+\P_{ab}\si+\La g_{ab} \\ \nd_a\La-\P_{ad}\mu^d
\end{pmatrix},
\end{align*}
just like the higher dimensional case. Recall that in dimensions $n>2$, there is a 1-1 correspondence between solutions of the conformally Einstein equation and parallel sections of the standard tractor bundle with nowhere vanishing scale $\si$ (see for instance \cite{AlmObst}). We would like to establish a similar correspondence on M\"obius surfaces. However, we have seen from the prolongation of the conformally Einstein equation (\ref{ce}) that an additional term involving $U_a$ appears in the closed system, so the standard tractor connection is not the right object in the correspondence. We have however in the flat case, where $U_a=0$, that  
\begin{proposition}\label{meflat}
There is a 1-1 correspondence between solutions of the conformal Einstein equation on flat M\"obius surfaces and parallel sections $\I_A$ of the standard tractor bundle with $\si=\X^A\I_A$ non-vanishing. 
\end{proposition}
\begin{proof}
Suppose the flat M\"obius surface admits a solution to (\ref{ce}). In a particular conformal scale, a parallel section of the tractor bundle $\cE_A$ is given by
\begin{align*}
\I_A=\begin{pmatrix}
\si \\ \nd_a\si \\ -\frac{1}{2}(\De \si+K\si)
\end{pmatrix}
\end{align*}
and we find from the prolonged system of (\ref{ce}) that 
\begin{align*}
\nd_a\I_B=
\nd_a\begin{pmatrix}
\si \\ \nd_b\si \\ -\frac{1}{2}(\De \si+K\si)
\end{pmatrix}
=\begin{pmatrix}
\nd_a\si-\nd_a\si \\ \nd_a\nd_b\si+\P_{ab}\si-\frac{1}{2}(\De \si+K\si)g_{ab} \\ \nd_a(-\frac{1}{2}(\De \si+K\si))-\P_{ad}\nd^d\si
\end{pmatrix}
=0,
\end{align*}
where the equation on the bottom slot holds as a differential consequence of (\ref{ce}) in the flat setting. Conversely given a parallel section of the standard tractor connection with nowhere vanishing scale, $\si=\X^A\I_A$ defines a solution to (\ref{ce}).  
\end{proof}
The prolongation connection for the conformal Einstein equation on general (possibly non-flat) M\"obius surfaces is given by 
\begin{align}\label{prolongcon}
D_a\T_B=D_a\begin{pmatrix}
\si \\ \mu_b \\ \La
\end{pmatrix}=&\begin{pmatrix}
\nd_a\si-\mu_a \\ \nd_a\mu_b+\P_{ab}\si+\La g_{ab} \\ \nd_a\La-\P_{ad}\mu^d+\frac{1}{2}U_a\si
\end{pmatrix}=&\begin{pmatrix}
\nd_a\si-\mu_a \\ \nd_a\mu_b+\P_{ab}\si+\La g_{ab} \\ \nd_a\La-\P_{ad}\mu^d
\end{pmatrix}+\begin{pmatrix}
0 \\0 \\ \frac{1}{2}U_a\si
\end{pmatrix}\\
=&\nd_a\T_B+\frac{1}{2}U_a\T^A\X_A\X_B.\nonumber
\end{align}
Analogous to the proof of Proposition \ref{meflat}, we obtain
\begin{proposition}
There is a 1-1 correspondence between solutions of the conformal Einstein equation on M\"obius surfaces and sections of the standard tractor bundle with $\si=\X^A\I_A$ non-vanishing that are parallel with respect to the prolongation connection given by (\ref{prolongcon}). 
\end{proposition}
In the case where the M\"obius surface is flat, the prolongation connection agrees with the standard tractor connection, which has zero curvature, and we recover Proposition \ref{meflat}.

\section{Local obstructions to the existence of conformal Einstein metrics on generic M\"obius surfaces and examples}\label{ceobst}
In this section we observe how the additional curvature term appearing in the closed system affects deriving obstructions. We proceed by differentiating the closed system to yield algebraic constraints on the conformal-to-Einstein system. 
Differentiating equation (\ref{ce3}) once more gives
\begin{align*}
\nd_b\nd_a\La=-\frac{1}{2}(\nd_bU_a)\si-\frac{1}{2}U_a\mu_b+(\nd_b\P_a{}^d)\mu_d+\P_a{}^d(-\La\de_{bd}-\P_{bd}\si)
\end{align*} 
and contracting with the inverse volume form $\epsilon^{ab}$, we obtain 
\begin{equation}\label{const1}
Y_a\mu^a+\mu\si=0,
\end{equation}
where $\mu:=\frac{\nd_aY^a}{2}$ (see \cite{thesis} for the conformal weight and conformal transformation of $\mu$).
Equation (\ref{const1}) is the first constraint of the conformal-to-Einstein system. Differentiating equation (\ref{const1}) once more and using the closed system gives
\begin{align}\label{const2}
(\nd_eY_d)\mu^d+(\nd_e\mu)\si+\mu\mu_e-\La Y_e-\P_e{}^dY_d\si=0. 
\end{align}
We contract the free index with $U^e$ and get
\begin{align*}
(U^e\nd_eY_d)\mu^d+(U^e\nd_e\mu)\si+\mu(\mu_eU^e)-\P_{ed}U^eY^d\si=0. 
\end{align*}
Now from the conformal transformation rule ($Y_a$, $U_a$ have weight $-2$)
\begin{align*}
\widehat{U^e\nd_eY_d}=U^e\nd_eY_d-3\Up_eU^eY_d+U_d\Up^cY_c,
\end{align*}
we observe that the 1-form $V_d$ given by
\[
V_d:=U^e\nd_eY_d+\mu U_d-3\phi Y_d
\]
is conformally invariant of weight $-6$, where here $\phi:=\frac{\nd_aU^a}{2}$ is also introduced in \cite{thesis}. 
Hence, applying (\ref{const2}) and (\ref{const1}) we get
\begin{align*}
V_d\mu^d=(U^e\nd_eY_d+\mu U_d-3\phi Y_d)\mu^d=-(U^e\nd_e\mu)\si+\P_{ed}U^eY^d\si+3\phi\mu\si.
\end{align*}
Let us call 
\[
k:=\P_{ed}U^eY^d-U^e\nd_e\mu+3\phi\mu.
\]
This is a scalar density of weight $-4$ and under conformal rescaling, $\hat k=k+\Up_cV^c$,
and we have
\begin{equation}\label{const3}
V_d\mu^d=k\si.
\end{equation}
Combining equations (\ref{const1}) and (\ref{const3}), and assuming that $U_dV^d \neq 0$, we can solve for $\mu^d$ to obtain 
\begin{align*}
\mu^d=-\frac{\mu \epsilon^{ad}V_a\si}{U_cV^c}+\frac{k \si U^d}{U_cV^c}=\frac{1}{U_cV^c}\left(\mu \epsilon^{da}V_a+k U^d\right)\si.
\end{align*} 
We now have to consider the cases where $U_aV^a=0$ and $U_aV^a \neq 0$. This motivates the following:
\begin{definition}
A non-flat M\"obius surface is called generic if $U_aV^a \neq 0$, and non-generic if $U_aV^a=0$ (the set where $U_aV^a \neq 0$ is open and locally we can restrict to a neighbourhood so that $U_aV^a \neq 0$). 
\end{definition}
On non-flat generic M\"obius surfaces, we can define the 1-form
\begin{align}\label{solk}
K_a:=\frac{1}{U_cV^c}\left(\mu \epsilon_a{}^{b}V_b+k U_a\right).
\end{align} 
Then $\mu_a=K_a\si$, 
and substituting this into (\ref{cen}) we obtain
\begin{align*}
\nd_a\mu_b=(\nd_aK_b)\si+\mu_aK_b=-\La g_{ab}-\P_{ab}\si, 
\end{align*} 
which gives
\begin{align*}
(\nd_aK_b)\si+K_aK_b\si+\P_{ab}\si=-\La g_{ab}, 
\end{align*} 
or that
\[
\La=-\frac{1}{2}(\nd_cK^c+K_cK^c+K)\si, 
\]
so that 
\begin{align*}
(\nd_aK_b)\si+K_aK_b\si+\P_{ab}\si=\frac{1}{2}(\nd_cK^c+K_cK^c+K)\si g_{ab}, 
\end{align*} 
and assuming $\si$ is non-zero, we obtain
\begin{align*}
\nd_aK_b+K_aK_b+\P_{ab}=\frac{1}{2}(\nd_cK^c+K_cK^c+K) g_{ab}. 
\end{align*} 
We have the following proposition
\begin{proposition}\label{ceeab}
Let $(M^2,[g],[\P])$ be a non-flat generic M\"obius surface (i.e.\ with $U_aV^a \neq 0$). Suppose it admits a solution to the conformal-to-Einstein equation. Then the following tensor obstruction $E_{ab}$ given by
\begin{equation}\label{eab}
E_{ab}=\nd_aK_b+K_aK_b+\P_{ab}-\frac{1}{2}(\nd_cK^c+K_cK^c+K) g_{ab}
\end{equation}
must vanish, where $K_a$ is given by (\ref{solk}). Conversely, suppose the tensor $E_{ab}$ given by (\ref{eab}) 
vanishes for the $1$-form $K_a$ given by (\ref{solk}) on a non-flat generic M\"obius surface $(M^2,[g],[\P])$. Then $\nd_{[a}K_{b]}=0$, and taking $K_a=\nd_a\log \si$ for some function $\si$, we find that there exists a solution to the conformal-to-Einstein equation on $M^2$.   
\end{proposition}
We shall discuss the non-generic case where $U_aV^a=0$ in Section \ref{nongeneric}. We now give $2$ examples of non-flat generic M\"obius structures on the Euclidean plane $\R^2$ (with $K=0$), one with vanishing obstruction and one without. 
\begin{example}
The first example will be the M\"obius structure on $\R^2$ given by 
\begin{align*}
\P=&\P_{11}dxdx+2\P_{12}dxdy+\P_{22}dydy\\
=&\left(y-x+\frac{1}{2}y^4-\frac{1}{2}x^4\right)dxdx-2x^2y^2dxdy+\left(x-y+\frac{1}{2}x^4-\frac{1}{2}y^4\right)dydy.
\end{align*}
For this M\"obius structure, we find that 
$K_a$ remarkably simplifies to give
\[
K=K_1dx+K_2dy=x^2dx+y^2dy
\]
and
\begin{align*}
E_{ab}=\partial_aK_b+K_aK_b+\P_{ab}-\frac{1}{2}(\partial_cK^c+K_cK^c) g_{ab}=0,
\end{align*}
so that the obstruction vanishes and taking $\si=e^{\frac{x^3+y^3}{3}}$ gives us a solution to the conformal-to-Einstein equation.
\end{example}
\begin{example}  
The second example will be the M\"obius structure given by $\P_{ab}=x_{(a}\epsilon_{b)c}x^c$ on $\R^2$. We shall show that this M\"obius structure admits no solution to (\ref{ce}) by showing that the tensor obstruction does not vanish. Let $x^a$ be standard local coordinates on $\R^2$ so that $\partial_ax_b=\de_{ab}$. A computation shows that 
\[
\partial_a\P_{bc}=\de_{a(b}\epsilon_{c)d}x^d+x_{(b}\epsilon_{c)a},
\] 
from which we obtain
\[
Y_c=2\epsilon^{ab}\partial_a\P_{bc}=-4x_c, \qquad U_c=-4\epsilon_{ca}x^a.
\]
Further computations of the various quantities give us 
\begin{align*}
K_a=&\frac{1}{4x_cx^c}\left((x_cx^c)^2\epsilon_{ab}x^b-4x_a\right).
\end{align*}
We then find (with the aid of MAPLE) that
\begin{align*}
E_{ab}=\partial_aK_b+K_aK_b+\P_{ab}-\frac{1}{2}(\partial_cK^c+K_cK^c) g_{ab} \neq 0,
\end{align*}
and so we conclude that this M\"obius structure admits no solution to the conformal-to-Einstein equation. Also observe that $\partial_{[a}K_{b]}\neq 0$ for this M\"obius structure. 
\end{example}

\section{Non-generic M\"obius surfaces and examples}\label{nongeneric}
Here we examine the situation of the conformal-to-Einstein equation on non-flat non-generic M\"obius surface, namely the case where $Y_a \neq0$ and $U_aV^a=0$. One obstruction can be computed quite readily; since $U_aV^a=0$, $V^a=fY^a$ for some scalar density $f$ of weight $-4$. Under the conditions that both $Y_a \neq 0$ and $U_aV^a=0$ hold on the M\"obius structure, the quantity $f$ is invariantly defined but it is not a M\"obius invariant in the classical sense since it is firstly rational (and not polynomial) in the jets of the conformal structure and secondly defined only on the subclass of M\"obius structure for which $Y_a \neq 0$ and the invariant $U_aV^a=0$.
Then 
\begin{align*}
k\si=V_a\mu^a=fY_a\mu^a=-f\mu\si, 
\end{align*}
so that $k+f\mu$ is a conformally invariant obstruction to conformal-to-Einstein on non-generic M\"obius surfaces. 
For example, the M\"obius structure on $\R^2$ given by
\[
\P=\P_{11}dx^2+2\P_{12}dxdy+\P_{22}dy^2=\left(\frac{x^2}{2}-\frac{y^2}{2}\right)dx^2+\left(\frac{y^2}{2}-\frac{x^2}{2}\right)dy^2
\]
has $U_aV^a=0$ as we compute and find
\begin{align*}
Y^a=\begin{pmatrix}2y\\-2x\end{pmatrix},\quad
U^a=\begin{pmatrix}-2x\\-2y\end{pmatrix},\quad 
V^a=\begin{pmatrix}8y\\-8x\end{pmatrix}.
\end{align*}
Here $f=4$, $\mu=0$, $k=4xy(y^2-x^2)$, so that the obstruction $k+f\mu=4xy^3-4x^3y\neq 0$ and we can conclude that the M\"obius structure does not admit a conformal-to-Einstein scale. Ideally we would like to find a complete set of obstructions to characterise the non-generic setting in a similar way as the generic setting (the content of Proposition \ref{ceeab}), but this case does not seem to be amenable to extracting obstructions. We are able to show however that provided two additional invariant conditions hold, the conformal Einstein equation on non-generic M\"obius surfaces reduces to a second order linear homogeneous ODE. We first need the following:  
\begin{lemma}
Let $(M^2,[g],[\P])$ be a non-flat non-generic M\"obius surface. Let $\rho:=Y^aY_a=U^aU_a$ be a conformally invariant density of weight $-6$ (the quantity $\rho$ is also introduced and used in \cite{sfmew}). Then the conformally invariant condition
\begin{equation}\label{murho}
\frac{Y^a\nd_a \rho}{6\rho}=\mu
\end{equation}
must hold. 
\end{lemma}
\begin{proof}
From the definition of non-generic M\"obius surface, we obtain
\begin{align*}
0=U^aV_a=&U^a(U^b\nd_bY_a+\mu U_a-3\phi Y_a)\\
=&-Y^aU^b\nd_bU_a+\mu \rho\\
=&-Y^aU^b(\nd_bU_a-\nd_aU_b)-Y^aU^b\nd_aU_b+\mu \rho\\
=&-Y^aU^b\epsilon_{ba}\epsilon^{cd}(\nd_cU_d)-\frac{Y^a\nd_a\rho}{2}+\mu \rho\\
=&Y^aU^b\epsilon_{ba}(\nd_cY^c)-\frac{Y^a\nd_a\rho}{2}+\mu \rho\\
=&2\rho\mu-\frac{Y^a\nd_a\rho}{2}+\mu \rho\\
=&-\frac{Y^a\nd_a\rho}{2}+3\mu \rho,
\end{align*}
and so equation (\ref{murho}) must hold (since the surface is non-flat, $\rho$ is non-zero). 
\end{proof}
\begin{proposition}
Let $(M^2,[g],[\P])$ be a non-flat non-generic M\"obius surface. Then the 1-form 
\[
\om_a:=\rho^{-\frac{1}{3}}U_a
\]
is closed and so is locally the gradient of some function $\eta$, so that $\om_a=\nd_a\eta$. 
\end{proposition}
\begin{proof}
We find that
\begin{align*}
\epsilon^{ab}\nd_a\om_b=&\epsilon^{ab}(-\frac{1}{3}\rho^{-\frac{4}{3}}\nd_a\rho U_b+\rho^{-\frac{1}{3}}\nd_aU_b)\\
=&\epsilon^{ab}\left(-\frac{\nd_a\rho}{3\rho} U_b+\nd_aU_b\right)\rho^{-\frac{1}{3}}\\
=&\left(\frac{Y^a\nd_a\rho}{3\rho}-2\mu\right)\rho^{-\frac{1}{3}}\\
=&0,
\end{align*}
where the last equality holds by equation (\ref{murho}). Hence $\om_a$ is closed. 
\end{proof}
Observe that $Y^a\nd_a\eta=0$, in other words the directional derivative of $\eta$ along $Y^a$ is zero. The integral curves of $Y^a$ form the characteristic lines of the equation $Y^a\nd_a\eta=0$ and $\eta$ is constant along those lines. On a non-generic M\"obius surface satisfying (\ref{ce0}), we have from (\ref{const1}) that
\begin{align*}
0=Y_a\mu^a+\mu\si=Y^a\mu_a+\frac{Y^a\nd_a \rho}{6\rho}\si=Y^a\left(\nd_a\si+\frac{\nd_a \rho}{6\rho}\si\right)
\end{align*}
must hold. This implies
\begin{align*}
Y^a\nd_a\log(\si\rho^{\frac{1}{6}})=0,
\end{align*}
and so $\log(\si\rho^{\frac{1}{6}})$ and $\eta$ are functionally dependent and viewing $\eta$ as a coordinate function we can express $\si=\rho^{-\frac{1}{6}}e^{s(\eta)}$ for some function $s$ (up to some constant multiple).
Differentiating $\si$ again gives
\begin{align}\label{mua}
\mu_a=\nd_a\si=\left(s_{\eta}\nd_a\eta-\frac{\nd_a \rho}{6\rho}\right)\si
\end{align}
and substituting this formula for $\mu_a$ into (\ref{cen}) gives us a second-order differential equation $s$ has to satisfy if (\ref{ce0}) holds. We can compute this explicitly. We now show
\begin{proposition}\label{nongenericde}
Let $(M^2,[g],[\P])$ be a non-flat non-generic M\"obius surface. Then the conformal-to-Einstein equation (\ref{ce}) is equivalent to a second order differential equation given by (\ref{sode}). If in addition, 2 further M\"obius invariant conditions given later by (\ref{constode1}) and (\ref{constode2}) hold, the equation reduces to a second order linear homogeneous ODE.  
\end{proposition}
\begin{proof}
Differentiating (\ref{mua}) we find that
\begin{align*}
\nd_b\nd_a\si=&\left(s_{\eta\eta}\nd_b \eta\nd_a\eta+s_{\eta}\nd_b\nd_a\eta-\frac{\nd_b\nd_a \rho}{6\rho}+\frac{\nd_b\rho \nd_a\rho}{6\rho^2}\right)\si\\
+&\left(s_{\eta}\nd_a\eta-\frac{\nd_a \rho}{6\rho}\right)\left(s_{\eta}\nd_b\eta-\frac{\nd_b \rho}{6\rho}\right)\si,
\end{align*}
so that substituting this expression into (\ref{ce}) gives
\begin{align}\label{subode}
&\left(s_{\eta\eta}\nd_b \eta\nd_a\eta+s_{\eta}\nd_b\nd_a\eta-\frac{\nd_b\nd_a \rho}{6\rho}+\frac{\nd_b\rho \nd_a\rho}{6\rho^2}\right)\si\nonumber\\
+&\left(s_{\eta}\nd_a\eta-\frac{\nd_a \rho}{6\rho}\right)\left(s_{\eta}\nd_b\eta-\frac{\nd_b \rho}{6\rho}\right)\si+\P_{ab}\si+\La g_{ab}=0.
\end{align}
Contracting (\ref{subode}) with $Y^aU^b$ gives
\begin{align}\label{subodeuy}
&\left(s_{\eta}Y^aU^b\nd_b\nd_a\eta-\frac{Y^aU^b\nd_b\nd_a \rho}{6\rho}+\frac{U^b\nd_b\rho Y^a\nd_a\rho}{6\rho^2}\right)\si \nonumber\\
+&\left(-\frac{Y^a\nd_a \rho}{6\rho}\right)\left(s_{\eta}U^b\nd_b\eta-\frac{U^b\nd_b \rho}{6\rho}\right)\si+\P_{ab}Y^aU^b\si=0.
\end{align}
Observing that $U^a\nd_a\eta=\rho^{\frac{2}{3}}$, we find that
\begin{align*}
Y^aU^b\nd_a\nd_b\eta=&-(U^c\nd_cY^a)\nd_a\eta\\
=&-(fY^a+3\phi Y^a-\mu U^a)\nd_a \eta\\
=&\mu\rho^{\frac{2}{3}},
\end{align*}
and hence equation (\ref{subodeuy}) simplifies to give 
\begin{align*}
\P_{ab}Y^aU^b-\frac{Y^aU^b\nd_b\nd_a \rho}{6\rho}+\frac{7U^b\nd_b\rho Y^a\nd_a\rho}{36\rho^2}=0
\end{align*}
after dividing throughout by $\si$. A computation shows that the expression on the left hand side equals $k+f\mu$, which must vanish if equation (\ref{ce0}) holds.  
Let 
\begin{align*}
\tilde Q=\P_{ab}Y^aY^b-\frac{Y^aY^b\nd_b\nd_a \rho}{6\rho}+\frac{7Y^b\nd_b\rho Y^a\nd_a\rho}{36\rho^2}. 
\end{align*}
We find that under a conformal rescaling, $\widehat {\tilde Q}={\tilde Q}+\frac{1}{2}\Up_c\Up^c\rho-\frac{\Up^c\nd_c\rho}{6}$.
Contracting equation (\ref{subode}) with $Y^aY^b$ gives
\begin{align*}
&\left(s_{\eta}Y^bY^a\nd_b\nd_a\eta-\frac{Y^bY^a\nd_b\nd_a \rho}{6\rho}+\frac{7Y^b\nd_b\rho Y^a\nd_a\rho}{36\rho^2}\right)\si+\P_{ab}Y^aY^b\si+\La \rho\\
=&\left(s_{\eta}Y^bY^a\nd_b\nd_a\eta\right)\si+\tilde Q\si+\La \rho\\
=&0
\end{align*}
and we find that
\begin{align*}
Y^bY^a\nd_b\nd_a\eta=-(Y^c\nd_cY^a)\nd_a\eta=-(f+\phi)U^a\nd_a\eta=-(f+\phi)\rho^{\frac{2}{3}},
\end{align*}
so that
\begin{align*}
\left(s_{\eta}Y^bY^a\nd_b\nd_a\eta\right)\si+\tilde Q\si+\La \rho=-\left(s_{\eta}(f+\phi)\rho^{\frac{2}{3}}\right)\si+\tilde Q\si+\La \rho=0.
\end{align*}
Contracting equation (\ref{subode}) with $U^aU^b$ and making the appropriate substitution gives
\begin{align*}
\left(s_{\eta \eta} \rho^{\frac{4}{3}}+s_{\eta}s_{\eta}\rho^{\frac{4}{3}}+s_{\eta}\left(U^bU^a\nd_b\nd_a\eta-\frac{U^b\nd_b \rho}{3\rho}\rho^{\frac{2}{3}}\right)\right) \sigma+\tilde R\si+\La \rho=0,
\end{align*}
where
\begin{align*}
\tilde R=\P_{ab}U^aU^b-\frac{U^aU^b\nd_b\nd_a \rho}{6\rho}+\frac{7U^b\nd_b\rho U^a\nd_a\rho}{36\rho^2}. 
\end{align*}
We find that under a conformal rescaling, $\widehat {\tilde R}={\tilde R}+\frac{1}{2}\Up_c\Up^c\rho-\frac{\Up^c\nd_c\rho}{6}$.
A computation shows that 
\begin{align*}
U^c\nd_cU^a=\frac{\nd^a\rho}{2}-2Y^a\mu, 
\end{align*}
so that from 
\begin{align*}
(U^c\nd_cU^a)\nd_a\eta+U^cU^a\nd_c\nd_a\eta=\frac{2}{3}\rho^{-\frac{1}{3}}U^c\nd_c\rho
\end{align*}
we have
\begin{align*}
U^cU^a\nd_c\nd_a\eta=\frac{2}{3}\rho^{-\frac{1}{3}}U^c\nd_c\rho-\frac{1}{2}\rho^{-\frac{1}{3}}U^a\nd_a\rho=\frac{1}{6}\rho^{-\frac{1}{3}}U^a\nd_a\rho, 
\end{align*}
and therefore
\begin{align*}
\left(s_{\eta \eta} \rho^{\frac{4}{3}}+s_{\eta}s_{\eta}\rho^{\frac{4}{3}}-s_{\eta}\frac{U^b\nd_b \rho}{6\rho}\rho^{\frac{2}{3}}\right) \sigma+\tilde R\si+\La \rho=0.
\end{align*}
But
\begin{align*}
\frac{U^c\nd_c\rho}{6\rho}=\frac{f}{3}+\phi,
\end{align*}
so that we have
\begin{align*}
\left(s_{\eta \eta} \rho^{\frac{4}{3}}+s_{\eta}s_{\eta}\rho^{\frac{4}{3}}-s_{\eta}\left(\phi+\frac{f}{3}\right)\rho^{\frac{2}{3}}\right) \sigma+\tilde R\si+\La \rho=0.
\end{align*}
Substituting $\La \rho=s_{\eta}\rho^{\frac{2}{3}}\left(\phi+f\right) \si-\tilde Q \sigma$, we obtain 
\begin{align*}
\left(s_{\eta \eta} \rho^{\frac{4}{3}}+s_{\eta}s_{\eta}\rho^{\frac{4}{3}}+s_{\eta}\frac{2f}{3}\rho^{\frac{2}{3}}\right) \sigma+(\tilde R-\tilde Q)\sigma=0.
\end{align*}
Dividing throughout by $\si \rho^{\frac{4}{3}}$ gives the second order differential equation
\begin{align*}
s_{\eta \eta}+s_{\eta}s_{\eta}+s_{\eta}\frac{2f}{3}\rho^{-\frac{2}{3}}+(\tilde R-\tilde Q)\rho^{-\frac{4}{3}}=0.
\end{align*}
The differential equation can be made linear by the substitution $s_{\eta}=\frac{\xi_\eta}{\xi}$, and we obtain
\begin{align}\label{sode}
\xi_{\eta \eta}+\frac{2f}{3}\rho^{-\frac{2}{3}}\xi_{\eta}+(\tilde R-\tilde Q)\rho^{-\frac{4}{3}}\xi=0.
\end{align}
This is a second order homogeneous ODE provided that the coefficient functions $\frac{2f}{3}\rho^{-\frac{2}{3}}$ and $(\tilde R-\tilde Q)\rho^{-\frac{4}{3}}$ are functions of $\eta$ only. (We remark that the coefficients are invariant under a conformal rescaling). 
To check that the coefficients are functions of $\eta$, we have to check whether both equations
\begin{align}\label{constode1}
Y^a\nd_a\left(\frac{2f}{3}\rho^{-\frac{2}{3}}\right)&=0, \\ \label{constode2}
Y^a\nd_a\left((\tilde R-\tilde Q)\rho^{-\frac{4}{3}}\right)&=0,
\end{align}
hold, i.e.\ the derivative of the coefficients in the direction orthogonal to $\eta$ is zero. Hence provided that (\ref{constode1}) and (\ref{constode2}) both hold in the non-generic case, equation (\ref{ce0}) reduces to a second order linear homogeneous ODE and it will admit $2$ linearly independent solutions. 
\end{proof}
\begin{example}
An example where (\ref{ce0}) is satisfied on non-flat non-generic M\"obius surfaces is the Euclidean plane $\R^2$ (with $K=0$) with the M\"obius structure on $\R^2$ given by 
\begin{align*}
\P=&\P_{11}dxdx+2\P_{12}dxdy+\P_{22}dydy\\
=&\left(-a-2a^2x^2\right)dxdx+\left(2a^2x^2+a\right)dydy,
\end{align*}
where $a$ is a constant. 
For this M\"obius structure, we obtain
\begin{align*}
Y=&Y_1dx+Y_2dy=8a^2xdy,&\phi=&4a^2,& \rho=&64a^4x^2,\\
V=&V_1dx+V_2dy=-32a^4xdy,&\mu=&0,
\end{align*}
and we find that $U_aV^a=0$, $k=0$ (and so $k+f\mu=0$). Furthermore, we find that 
$Y^a\nd_a(\frac{2f}{3}\rho^{-\frac{2}{3}})=0$ and $Y^a\nd_a((\tilde R-\tilde Q)\rho^{-\frac{4}{3}})=0$ (since $\rho$, $f$, $\tilde R$ and $\tilde Q$ depend only on $x$). In this example, equation (\ref{ce0}) reduces to a second order ODE and the M\"obius surface admits two linearly independent solution to (\ref{ce0}). One is given by $\si=e^{ax^2}$, $\mu_a=2axe^{ax^2}dx$, $\La=-(a+2a^2x^2)e^{ax^2}$ and the other is given by $\si=\erf(\sqrt{2a}x)e^{ax^2}$, $\mu_a=\left(2ax\erf(\sqrt{2a}x)e^{ax^2}+2\sqrt{\frac{2a}{\pi}}e^{-ax^2}\right)dx$, $\La=-(a+2a^2x^2)\erf(\sqrt{2a}x)e^{ax^2}$, where $\erf(x)$ is the Gaussian Error function defined by
\[
\erf(x)=\frac{2}{\sqrt{\pi}}\int^x_0 e^{-t^2}dt.
\]
\end{example}
\begin{example}
Another example is the Euclidean plane $\R^2$ (with $K=0$) with the M\"obius structure on $\R^2$ given by 
\begin{align*}
\P=\P_{11}dxdx+2\P_{12}dxdy+\P_{22}dydy=-\frac{x}{2} dxdx+\frac{x}{2} dydy.
\end{align*}
For this M\"obius structure, we obtain
\begin{align*}
Y=&Y_1dx+Y_2dy=dy,&\phi=&0,& \rho=1,\\
V=&V_1dx+V_2dy=0,&\mu=&0,& f=0,\\
U=&U_1dx+U_2dy=dx=d\eta, &\tilde Q=\P_{22}=&\frac{x}{2},  & \tilde R=\P_{11}=-\frac{x}{2}. 
\end{align*}
Equation (\ref{sode}) reduces to 
\begin{align*}
\xi_{x x}-x \xi=0,
\end{align*}
and the solutions to this second order ODE are given by the Airy functions of the first and second kind, denoted $Ai(x)$ and $Bi(x)$. 
\end{example}
\section{Dimension of the kernel of conformal-to-Einstein operator}\label{dimker}
Here following David Calderbank's suggestion, we find the dimension of the kernel of conformal-to-Einstein operator on M\"obius surfaces. The maximal dimension of $4$ is obtained in the flat setting, as we shall show in Subsection \ref{flatsol}. In Subsection \ref{generalsol} we obtain partial results on the dimension of the kernel in the non-flat setting. We show that the dimension of $1$ is attained in the generic case and the dimension of $2$ is attained in the non-generic second order ODE case. 
\subsection{The flat case}\label{flatsol}

In the flat setting, on $\R^2$ equipped with the flat metric $\de_{ab}$ and standard coordinates $x^a=(x^1,x^2)$ with the Rho tensor associated to the M\"obius structure given by $\P_{ab}=0$ (since the Gauss curvature $K=0$), the tractor connection reduces to 
\begin{align*}
\partial_a\I_B=\partial_a\begin{pmatrix}
\si \\ \mu_b \\ \La
\end{pmatrix}=&\begin{pmatrix}
\partial_a\si-\mu_a \\ \partial_a\mu_b+\La \de_{ab} \\ \partial_a\La
\end{pmatrix}=\begin{pmatrix}
0 \\ 0 \\ 0 
\end{pmatrix}.
\end{align*}
We find that $\La=C$ where $C$ is some constant, and $\mu_b=-Cx_b+B_b$ where $B_b$ is some constant 1-form. Then $\si=-\frac{C}{2}x_bx^b+B_bx^b+A$, and we find that the solutions are determined by $4$ constants and so the dimension of the solution space of the conformal-to-Einstein operator is $4$. The parallel sections of the tractor bundle on $\R^2$ are therefore given by
\begin{align*}
\I_A=
\begin{pmatrix}
-\frac{C}{2}x_bx^b+B_bx^b+A\\ -Cx_a+B_a \\ C
\end{pmatrix}
\end{align*}
which is unique up to some constant multiple. 
\subsection{The non-flat case}\label{generalsol}
For the non-flat case, in the generic setting we have $\dim \ker \cD=1$ from Proposition \ref{ceeab}. In the non-generic setting it is also seen that provided the two constraints (\ref{constode1}) and (\ref{constode2}) hold, equation (\ref{ce0}) reduces to a second order homogeneous linear ODE and in this setting $\dim \ker \cD=2$. The remaining cases where either of the two constraints (\ref{constode1}) or (\ref{constode2})  does not hold will require a lot more involved and intensive computation to investigate but we conjecture in that in this case, if any examples exist, the dimension of the kernel is $1$. We do not have any examples of the non-generic setting where any of the constraints (\ref{constode1}) or (\ref{constode2}) does not hold and we do not know whether the class of such solutions is empty or not.     
We therefore have
\begin{proposition}
Let $(M^2,[g],[\P])$ be a M\"obius surface, and $\cD_{ab}:\cE[1]\mapsto \cE_{(ab)_\circ}[1]$ be the conformal-to-Einstein operator given by (\ref{ceop}). Then $\dim \ker \cD=4$ if the M\"obius structure is flat and otherwise $\dim \ker \cD=1$ if the M\"obius structure is non-flat and generic, or $\dim ker \cD=2$ if the conformal-to-Einstein equation reduces to a second order ODE (\ref{sode}) in the non-flat and non-generic case (this is under the assumptions that both (\ref{constode1}) and (\ref{constode2}) hold). 
\end{proposition} 

\section{Outlook}\label{outlook}
A future research direction is to look at almost M\"obius-Einstein scales in the same vein as \cite{AlmObst} and \cite{Gal}. Let $(M^2,[g],[\P])$ be a M\"obius surface equipped with a non-zero parallel section $\I_A$ (with respect to the prolongation connection) of the standard tractor bundle. The scale singularity of $\si$ given by $\Si=\{p\in M^2|\si(p)=0\}$ can either be codimension $1$, which is a curve on the Riemann surface, or codimension $2$, in which case $\Si$ is possibly a collection of isolated singularities (we refer to the example where $\si$ is given by Airy functions). Such M\"obius surfaces are conformally Einstein away from the scale singularity set $\Si$. Another possible direction is to look at a conformally compact 3-manifold $M^3$ with the boundary at infinity equipped with a section of the standard tractor bundle parallel with respect to the prolongation connection, and investigate the extent that the geometry of the interior is determined by the boundary. This is closely related to the AdS/CFT correspondence in physics. 

\section{Acknowledgements}
The author will like to acknowledge the referees for comments and improvements of the article.

\end{document}